\documentclass[12pt,reqno]{amsart}
\usepackage{geometry}                		
\usepackage{graphicx}				
\usepackage{amssymb}
\usepackage{epstopdf}
\usepackage{datetime}
\usepackage{currfile} 
\allowdisplaybreaks

\usepackage{color}

\newtheorem{thm}{Theorem}[section]

\newtheorem{lem}[thm]{Lemma}
\newtheorem{prop}[thm]{Proposition}
\theoremstyle{definition}

\theoremstyle{remark}
\newtheorem{rem}{Remark}[section]

\numberwithin{equation}{section}
						
\numberwithin{equation}{section}

\allowdisplaybreaks

\newcommand{\mM}{\mathcal{M}}

\newcommand{\bv}{\textbf{v}}

\begin{document}

\title[Equivalent Formulations of Lagrangian Optimal Transport]
{A Unified Variational Framework for Optimal Transport with Lagrangian Costs}
\author{Hailiang Liu}
\address{Iowa State University, Mathematics Department, Ames, IA 50011} \email{hliu@iastate.edu}
\begin{abstract} 
We investigate optimal transport costs induced by general Lagrangian action functionals. Extending the classical Monge–Kantorovich and Benamou–Brenier theories, we derive a unified variational framework that connects several equivalent formulations of the induced transport cost, including Lagrangian, Eulerian, convex optimization, Hamilton–Jacobi dual, and Hamiltonian flow formulations. Under standard convexity assumptions on the Lagrangian, we establish the equivalence of these formulations through variational arguments and convex duality. The resulting optimality system reveals a natural Hamiltonian structure on the Wasserstein space, providing a direct link between optimal transport, Hamiltonian dynamics, and optimal control. The proposed framework extends the classical quadratic-cost theory to general Lagrangian costs and offers a unified perspective for the analysis of variational structures generated by action functionals.

\end{abstract}

\maketitle
\tableofcontents

\section{Introduction} 
Optimal transport  seeks the most efficient way to move mass from one distribution to another according to a prescribed transportation cost.  
In the classical quadratic-cost setting, this cost is closely related to the squared Wasserstein distance, and the resulting Wasserstein distance provides a genuine metric structure on the space of probability measures.  More general transportation costs arise naturally from variational principles, optimal control, mechanics, and other applications, but the resulting optimal transport value need not itself define a metric.  Our goal is to establish a unified variational framework for transport costs induced by the least action principle and to characterize them through several  equivalent  Lagrangian and Eulerian formulations.
These equivalent formulations provide complementary perspectives on the induced transport cost and reveal a richer relationship between the Lagrangian and Eulerian viewpoints.

Owing to its deep connections with  analysis, geometry, partial differential equations, probability, and optimization, optimal transport has become a central topic in modern mathematics \cite{Vi03, Vi08}.   More recently, optimal transport has attracted significant attention in computer science and machine learning due to its applications in statistical
modeling and data-driven learning \cite{Cu13,Ar17}.  In particular, optimal transport provides powerful metrics for comparing probability distributions and optimal transport maps for minimizing transportation costs. These tools have found broad applications in computer vision, including image retrieval, image  interpolation, and semantic correspondence, as well as in machine learning tasks such as domain adaptation, generative modeling, and variational inference; see for example, \cite{Ru98, PC19, CFT17, Ar17, Fl18} and references therein. 
Among the various optimal transport metrics, the Wasserstein distance has emerged as a fundamental tool for measuring discrepancies between probability measures.  The development of efficient numerical methods for optimal transport remains an active area of research \cite{PC19, Cu13}, particularly for large-scale problems and generalized transportation costs. Existing approaches include entropic regularization and Sinkhorn-type algorithms, PDE-based and variational discretizations, and dual optimization methods. In particular, Jacobs and Léger \cite{JL20, JLL21} introduced a fast ``back-and-forth'' method for computing optimal transport maps for a broad class of strictly convex costs, including quadratic and $p$-power costs. Their method alternates gradient-ascent steps in the two Kantorovich potentials and exploits efficient $c$-transforms, leading to favorable computational complexity and enabling high-resolution computations. These developments highlight the importance of efficient numerical methods beyond the classical quadratic-cost setting.


Classical optimal transport can be viewed from two complementary perspectives. The first is the {\bf Monge-Kantorovich formulation}, which seeks  an optimal transport map  (or, more generally, an optimal transport plan) between two probability measures. The second is the {\bf dynamical formulation}, which 
interprets optimal transport as the problem of finding geodesics in the space of probability measures. This dynamical viewpoint was introduced by Benamou and Brenier
\cite{BB00},  who show that the squared 2-Wasserstein distance between two probability densities  $\rho_0$ and $\rho_1$ on $D= \mathbb{R}^d$  
can be characterized as the minimum kinetic energy
$$
\int_0^1\int_{D} |v(t, x)|^2 \rho(t, x)dxdt
$$
subject to the continuity equation 
$$
\partial_t \rho  + \nabla_x \cdot (\rho v)=0, \quad \rho(0, \cdot)=\rho_0,  \; \rho(1, \cdot) =\rho_1.
$$
The resulting probability-measure-valued curves play a fundamental role in the theory of gradient flows in Wasserstein spaces;  see, for example, 
 \cite{Ot01, Vi03, AGS05, Mc97} and the references therein.

The classical formulation of optimal transport, on the other hand, seeks an optimal transport plan from the set
$$
\Pi(\mu_0, \mu_1)=\{\gamma\in \mathcal{P}(D\times D): (\pi_x)_{\#}\gamma=\mu_0, (\pi_y)_{\#}\gamma=\mu_1\},
$$
that is,  the set of probability measures on the product space $D \times D$  with prescribed marginals 
 $\mu_0$  and $\mu_1$. Given a transportation cost $c(x,y)$, the Kantorovich optimal transport problem is to solve
$$
\inf_{\gamma \in \Pi(\mu_0, \mu_1)}   \int_{D\times D}   c(x, y)d \gamma (x, y).
$$
Recall that the Benamou-Brenier formulation  characterizes  the  squared 2-Wasserstein distance, which corresponds to the quadratic transportation 
cost  
$$
c(x, y)=|x-y|^2.
$$ 
More generally,  transportation costs of the form 
 \begin{align}\label{h}
 c(x, y)=h(x-y),
 \end{align}
have been extensively studied; see,  for example,  \cite{GM96}. Here  $h: D \to \mathbb{R}$  is assumed to be continuous,  convex, and  even function.
Among these costs, the quadratic choice, $h(z)=|z|^2$, is by far the most widely studied, as it gives rise to the Wasserstein-2 distance and the Benamou--Brenier dynamical formulation.  We refer to \cite{KS84, Br91,Ev89, GM96, RR98} for the foundational theory and its many developments.


In this paper,  we focus on transport costs  induced by a general  {\bf Lagrangian cost}. The classical dynamical formulation of optimal transport interprets the Wasserstein -2 distance through curves of probability measures with minimal kinetic action, which correspond to constant-speed Wasserstein geodesics. In the present general Lagrangian setting, we instead consider action-minimizing curves associated with the prescribed Lagrangian. These curves need not be geodesics of a Riemannian metric unless the Lagrangian is quadratic in the velocity.  To avoid technical issues associated with boundary  condition and 
compactness, we work on the flat torus  $D=\mathbb{T}^d$. 
Given a Lagrangian 
$$
L: D\times \mathbb{R}^d \to \mathbb{R}, 
$$
we define the pointwise transportation cost by the least-action principle
\begin{align}\label{cl+} 
c_L(x, y)=\inf_{\omega} \left\{  \int_0^1 L(\omega(t), \dot \omega(t))dt,\quad \omega(0) = x, \; \omega(1) = y\right\}, 
\end{align}
for which the minimum is taken on the set of curves $\omega \in  C^1([0,1], D)$.  The corresponding transport cost between two probability measures 
$\rho_0$ to $\rho_1$   is then defined by
$$
C_L(\rho_0, \rho_1)=\inf_{\gamma\in \Pi(\rho_0, \rho_1)} \int_{D \times D} c(x, y)d\gamma(x, y).
$$
For the quadratic Lagrangian, $L(v)=\frac{1}{2}|v|^2$, we have 
$$
C_L(\rho_0, \rho_1)=\frac{1}{2}W_2^2(\rho_0, \rho_1),
$$ 
so the distance is 
$$
W_2(\rho_0, \rho_1) =\sqrt{2C_L(\rho_0, \rho_1)}. 
$$
For $p$-power costs, similarly, the cost is related to $W_p^p$, not $W_p$. 

We next develop several equivalent formulations of this transport cost, beginning with its Eulerian dynamical formulation. 
$$
\inf_{\rho, v} \int_0^1\int_D L(x, v(t, x)) \rho(t, x) dx dt,
$$ 
subject to the continuity equation 
$$
\partial_t \rho  + \nabla_x \cdot(\rho v)=0, \quad \rho(0, \cdot)=\rho_0,  \; \rho(1, \cdot) =\rho_1.
$$  
When the Lagrangian is convex with respect to the velocity variable, introducing the momentum variable $m=\rho v$ transforms the problem into a convex optimization problem.
By convex duality,  this formulation is equivalent to a variational problem over a potential $\phi$,
 which serves as the Lagrange multiplier associated with the continuity constraint. Formally, the corresponding optimality conditions consist of the coupled system
\begin{align*}
& \partial_t \rho +\nabla_x \cdot(\rho \nabla_p H(x, \nabla_x \phi))=0,\\ 
& \partial_t \phi +H(x, \nabla_x  \phi)
=0,
\end{align*}
together with the prescribed initial and terminal conditions for $\rho$.

This optimality system closely resembles the mean-field games system introduced by Lasry and Lions \cite{LL07}. 
More importantly, it admits a natural Hamiltonian interpretation as a Wasserstein--Hamiltonian flow  \cite{CLZ20}  
\begin{align*}
\partial_t \rho =\frac{\delta}{\delta \phi}\mathcal H(\rho, \phi), \quad \partial_t \phi=- \frac{\delta}{\delta \rho} \mathcal H(\rho, \phi),
\end{align*} 
with Hamiltonian  
$$
\mathcal H(\rho, \phi)= \int_D H(x, \nabla_x \phi)  \rho dx. 
$$
This formulation reveals that optimal transport induced by a general Lagrangian cost possesses an intrinsic infinite-dimensional Hamiltonian structure on the Wasserstein space of probability measures. The density $\rho$  represents the mean-field distribution of an ensemble of particles, while and the potential $\phi$ 
serves as its conjugate momentum.  The evolution is therefore governed by Hamilton's equations on the space of probability measures, providing a unified geometric framework that connects optimal transport, mean-field games, Hamiltonian dynamics, and Wasserstein geometry.

Optimal transport with Lagrangian costs has been investigated previously in several application-driven settings;  see, for example,  \cite{BB07, FF10, FGY11}. 
Our emphasis, however, is on the {\bf Eulerian viewpoint}. We develop several equivalent characterizations of the induced transport cost, derive the associated PDE optimality systems, and clarify their connections with optimal transport plans and maps.

The remainder of the paper is organized as follows. In Section 2, we establish the existence and uniqueness of minimizing curves defining the pointwise cost $c(x,y)$ and investigate its basic properties.  Section 3 presents several equivalent formulations of the induced transport cost and establishes their equivalence. In Section 4, we illustrate the framework through several representative examples of local costs arising from Lagrangians. Finally, Section 5 concludes with further discussions and remarks.


\section{Lagrangian Costs and Optimal Transport}
The choice of the transportation cost plays a fundamental role in optimal transport. While the Kantorovich formulation admits a minimizer under very mild assumptions on the cost function, the existence of an optimal transport map in the original Monge formulation is considerably more delicate and depends strongly on the structure of the cost.
 
 Two classical examples are the {\bf linear cost}
 $$
 c(x, y)=|x-y|,
 $$
  which corresponds to Monge's original formulation, and the {\bf quadratic cost}  
  $$
  c(x, y) = |x-y|^2,
  $$
 which gives rise to the 2-Wasserstein distance. The quadratic cost enjoys remarkable geometric and analytical properties, including deep connections with the Monge--Ampère equation, incompressible Euler equations, gradient flows, and diffusion equations such as the heat equation; see \cite{JKO98}. 
In this work we consider the more general  class of {\bf Lagrangian costs} 
 \begin{align}\label{cl} 
c_L(x, y)=\inf_{\omega} \left\{ \int_0^1 L(\omega(t), \dot \omega(t))dt,\quad \omega(0) = x, \; \omega(1) = y\right\}, 
\end{align} 
where the minimum is taken over all continuously differentiable curves connecting $x$ and $y$.  Throughout the paper, we work with a general Lagrangian $L(x,v)$,  relying only on its structural properties rather than any explicit expression. 

To ensure that the pointwise cost $c(x,y)$ is well defined,  we assume that $L$  is a  {\bf Tonelli Lagrangian}; namely $L(x,\cdot)$ is strictly convex and superlinear, 
and the associated Euler–Lagrange flow is complete \cite{Fa19, Da08}.
Under these standard assumptions, the direct method of the calculus of variations guarantees the existence of minimizing curves, so that the cost $c(x,y)$ is well defined.

The corresponding Kantorovich optimal transport problem is given by 
$$
C_L(\mu_0, \mu_1)= \inf_{\gamma \in \Pi(\mu_0, \mu_1)} \int_{D \times D}c_L(x, y)d\gamma(x, y).
$$
A measurable map $T:D \to D$ is said to transport  $\mu_0$ to $\mu_1$ if 
$$
\mu_1=T_\# \mu_0, 
$$
that is 
$$
\int_D f(y)d\mu_1(y) = \int_D f(T(x))d\mu_0(x)
$$
for every integrable functions $f$. Such a map induces the transport plan
$$
\gamma=(id, T)_\# \mu_0,
$$ 
and solves the Monge problem 
$$
{\rm inf}_{\mu_1=T_\# \mu_0} M[T]:=\int_D c_L(x, T(x))d\mu_0(x).
$$
Unlike the Kantorovich problem, the Monge problem may fail to admit a solution, since an optimal transport map does not always exist. One of the motivations for introducing transport plans is precisely to overcome this lack of existence while retaining the essential geometric and variational structure of the transport problem; see \cite{Ka48}. 

\section{Equivalent Formulations of the  Lagrangian Transport Cost} 

In this section, we present several equivalent formulations of the transport cost  $C_L(\mu_0, \mu_1)$. 
Throughout the paper, we assume that the involved measures are  absolutely continuous with respect to the Lebesque measure,   
$$
d\mu_i(x)=\rho_idx, \quad i=0, 1, 
$$
and identify each measure with its density $\rho_i$. \\

{\bf A geometric formulation}\\

For general nonlinear Lagrangians, we work with the Lagrangian action rather than invoking a Riemannian metric structure. 
Given a density $\rho$, consider a tangent vector $s$ satisfying 
$$
s + \nabla \cdot(\rho v)=0. 
$$
The Lagrangian $L(x, v)$ induces the following action-based functional on tangent directions 
define  the variational structure induced by the Lagrangian $L$: 
\begin{align*}
\mathcal{G}_\rho(s)
:=
\inf_{\substack{v: \ s+\nabla\cdot(\rho v)=0}}
\int_D \rho(x)L(x,v(x))dx. 
\end{align*} 
Here, the infimum is taken over all admissible velocity fields $v$ that generate the tangent vector $x=s$ through the continuity equation on manifold 
$$
\mM =\left\{ \rho \in \mathcal{P}(D): \quad \rho >0, \quad  \int_D \rho(x)dx=1\right\}, 
 $$
where $\mathcal{P}(D)$ denotes the set of probability densities on $D$. 
When $L(x,v)$ is quadratic in $v$, for example
$$
L(x,v)=\frac12 v^\top A(x)^{-1}v,
$$
the functional $\mathcal{G}_\rho(s, s)$ defines a Riemannian metric tensor on $\mM$. In the classical case $A(x)=I$, this reduces to the standard Wasserstein-2 Riemannian structure.

The corresponding action-based transport cost is defined by 
\begin{equation}\label{gr}
I_1: =\inf_{\rho}  
\int_0^1 \mathcal{G}_\rho \left(\partial_t \rho \right)dt, 
\end{equation}
where the infimum is preformed over all sufficiently smooth curves 
$$
\rho: [0, 1] \to  \mM, \quad \rho(0)=\rho_0, \quad \rho(1)=\rho_1.
$$
In the special case where $L$ is quadratic in the velocity, $\mathcal{G}_\rho$ defines a Riemannian metric tensor, and the corresponding cost can be interpreted as the squared length of a geodesic in the induced Riemannian structure.

{\bf A Lagrangian formulation} \\
Next, we consider  the Lagrangian description. Let 
$$
X: [0, 1]\times D \to D
$$
be a family of particle trajectories. 
The transport cost is given by 
\begin{equation}\label{xs}
I_2: = \inf_{X} 
\int_0^1\int_D L(X(t, x), \partial_t X(t, x))\rho_0(x)dx dt, 
\end{equation}
where the infimum is preformed over all sufficiently smooth maps satisfying the push-forward constraint
$$
X(1, \cdot)_{\#}\rho_0=\rho_1. 
$$

{\bf An Eulerian  formulation} \\
We now turn to the Eulerian description. 
Let  $\rho(t, x) \geq 0$ and $v(t, x)\in \mathbb{R}^d$ denote the density and velocity fields, respectively. They satisfy  the continuity equation 
\begin{equation}\label{tr}
\partial_t \rho +\nabla_x \cdot(\rho v)=0
\end{equation}
together with  the initial and final conditions
\begin{equation}\label{r01}
\rho(0)=\rho_0, \quad \rho(1)=\rho_1.
\end{equation}
The corresponding optimization problem is
\begin{equation}\label{rv-}
I_3: = \inf_{\rho, v} \int_0^1 \int_{D} \rho(t, x)L(x,  v(t, x))dx dt.
\end{equation}
Although this formulation is natural, the objective is generally not jointly convex in the variables $(\rho, v)$. \\

\noindent{\bf A convex momentum formulation} \\
Following the idea of Benamou and Brenier  \cite{BB00}  for the quadratic cost, we introduce the momentum variable 
$$
m= \rho v. 
$$
The continuity equation then becomes 
$$
\partial_t \rho +\nabla_x \cdot(m)=0,
$$
which is linear in the unknowns $(\rho, m)$. The optimization problem can be rewritten as
i
\begin{equation}\label{rm}
I_4: = \inf_{\rho, m} \int_0^1 \int_{D} \rho(t, x)L(x,  m(t, x)/\rho(t, x))dx dt
\end{equation}
subject to the continuity equation and the endpoint constraints.
The functional is understood as $0$ for $(\rho, m)=(0, 0)$ and $\infty$ for other than $\rho=0$ and $ m\not=0$.   
This is the standard convex extension used in optimal transport.

This formulation is particularly attractive because, under the convexity of the Lagrangian with respect to the velocity variable, the objective becomes convex in the variables  $(\rho,m)$.
Consequently, powerful tools from convex analysis and duality become available.

{\bf A convex dual representation}\\
To derive the dual formulation, we first establish the following characterization of the convex integrand.
\begin{lem} \label{lem3.1}
Let $H(x, p)$ be the Legendre transform of $L(x, q)$ with 
$$
H(x, p)=\sup_{q\in \mathbb{R}^d}\{p\cdot q - L(x, q)\}.
$$
Define  
 $$
 Q(x)=\left\{(a, b)\in \mathbb{R}\times \mathbb{R}^d: \quad a+H(x, b)\leq 0 \right\} \quad \text{ for each}\quad x\in D.
 $$
Then,  for every $(\tau, y) \in \mathbb{R} \times \mathbb{R}^d$,  
$$
\sup_{(a, b)\in Q(x)}
 ( a\tau  +b\cdot y ) = g(\tau, y),
 $$
 where 
 $$
 g(\tau, y) =\left\{ 
\begin{array}{ll}
\tau L(x, y/\tau), & \tau >0,\\
0, & \tau=0, y=0,\\
+\infty, & \tau=0, y\not=0 \; \text{or}\; \tau <0.
\end{array}
\right.
$$
\end{lem}
The proof follows directly from the Legendre transform and elementary convex duality. This lemma provides the key ingredient for deriving the dual Eulerian formulation presented in the next subsection. 
\begin{proof}  We consider three cases.  If $\tau>0$, then the constraint $a+H(x, b)\leq 0$ implies that the supermum is attained by taking  the largest admissible   value of $a$, namely 
take the maximal possible value of $a$ in the sup, and hence
$a=-H(x, b)$. Hence 
$$
\sup_{(a, b)\in Q(x)} (a\tau +b\cdot y) =
\sup_b [\left(-\tau H(x, b)+b\cdot y \right)=\tau \sup_b 
\left(b\cdot y/\tau  - H(x, b) \right)]=\tau L(x, y/\tau),
$$
where the last equality follows from the definition of the Legendre transform.  

If $(\tau, y)=(0, 0)$, then $a\tau +b\cdot y=0$ for every admissible pair  $(a, b)$, so the supremum is equal to zero.  

Next, suppose  $\tau=0$ and $y\not=0$.  Since $a$ does not contribute to the objective, for any $b$ 
one can choose $a$ sufficiently negative so that $a+H(x, b)\leq 0$.  Therefore,  
$$
\sup_{(a, b)\in Q(x)} \{a\tau+b\cdot y\}=\sup_b  b\cdot y=+\infty.
$$
Finally, if $\tau <0$,  choose $b=0$. Since $H(x,0)$ is bounded on the compact domain $D$, any sufficiently negative value of $a$ satisfies $a+H(x, 0)< 0 $. 
Because $\tau <0$, $a\tau \to \infty$ as $a\to -\infty$, which yields 
$$
\sup_{(a, b)\in Q(x)} \{a\tau+b\cdot y\}=\sup_a  a\tau = +\infty.
$$
This completes the proof. 
\end{proof}

The convex duality developed naturally leads to  the following variational formulation: 
 $$
 I_5:= \sup_{\phi} \int_{D } (\phi(1, x)\rho_1(x)-\phi(0, x)\rho_0(x))dx,
$$
subject to the Hamilton-Jacobi inequality  
$$
 \partial_t \phi + H(x, \nabla_x \phi) \leq 0. 
$$
We now show that all of the preceding formulations  are equivalent.

\begin{thm}\label{thm1}
Let $\rho_0, \rho_1 \in \mathcal P(D)$ be smooth probability densities.   Then the Lagrangian transport cost satisfies
\begin{equation}\label{rv}
C_L(\rho_0, \rho_1)=I_2 = I_3, 
\end{equation}
and, furthermore,  
$$
C_L(\rho_0, \rho_1)=I_4=I_5.
$$
\end{thm}
  \begin{proof} 
The proof proceeds in four steps: \\
{\bf 1.} ``Establish the inequality $I_3\geq I_2$, linking the Eulerian and Lagrangian formulations"; \\
{\bf 2.}  ``Reformulate the Eulerian problem $I_3$  as the convex optimization problem $I_4$; \\
{\bf 3.}   ``Derive the dual problem $I_5$  via Fenchel--Rockafellar duality";\\ 
{\bf 4.}  ``Show that $I_5$ coincides with the Kantorovich dual formulation, yielding  completing the equivalence of all variational characterizations." \\ 
 
\noindent{\bf Step 1.} \\
Let $(\rho, v)$ be an admissible pair  for Eulerian formulation (\ref{rv}) with $(\rho v)$ satisfying the continuity equation (\ref{tr}) together with the endpoint conditions 
 (\ref{r01}). 
Let $X(t, x)$ denote the flow generated by the velocity field $v$, namely, 
\begin{equation}\label{xv}
X(0, x)=x, \quad \partial_t X(t, x)=v(t, X(t, x)),
\end{equation}
for all $(t, x)\in [0, 1]\times D$.  
Since $(\rho,v)$ satisfies the continuity equation, the flow transports the initial density to the density at time $t$; that is,
$$
\rho(t) = X(t, \cdot)_{\#}\rho_0.
$$
Using the change-of-variables formula associated with the push-forward, we obtain
\begin{align*}
\int_0^1 \int_{D} \rho(t, x)L(x, v(t, x))dx dt  & = \int_0^1 \int_{D} \rho_0(x) L(X(t, x), v(t, X(t, x))dxdt \\
& = \int_0^1 \int_{D} \rho_0(x) L(X(t, x), \partial_t X(t, x))dxdt.
\end{align*}
By the definition of the Lagrangian cost,
$$
c_L(x, X(1, x))d \leq \int_0^1 L(X(t, x), \partial_t X(t, x))dt,
$$
and therefore 
$$
\int_0^1 \int_{D} \rho(t, x)L(x, v(t, x))dx dt  \geq  \int_{D} \rho_0(x)c_L(x, X(1, x))dx. 
$$
Since the transport map $X(1,\cdot)$ pushes $\rho_0$  forward to $\rho_1$ , 
the induced transport plan is admissible for the Kantorovich problem. Hence
$$
\int_{D} \rho_0(x)c_L(x, X(1, x))dx \geq C_L(\rho_0, \rho_1). 
$$
This proves that every admissible Eulerian pair has cost no smaller than the Lagrangian optimal transport cost, establishing the first implication. \\

\noindent{\bf  Step 2.} \\  
Since the functional is not convex in $(\rho, v)$, we introduce the moment variable   
$$
m=\rho v, 
$$
and define a convex functional 
\begin{align} \label{grm}
J(\rho, m)=\sup_{(\sigma, \xi)\in Q(x)} \int_{[0\; 1]\times D } ( \sigma \rho(t, x) +\xi \cdot m(t, x))dxdt.
\end{align}
By Lemma \ref{lem3.1}, 
\begin{align}\label{gl}
J(\rho, m) =  \int_{D}\int_0^1  L(x, m/\rho)\rho dt dx, 
\end{align}
which is the convex extension of the original action. 

In fact, 
\begin{align*} 
& \int_{D}\int_0^1 (\rho(t, x)\sigma +m\cdot \xi)dt dx \\
&   \qquad \leq \int_{D}\int_0^1 (- \rho(t, x) H(x, \xi) +m\cdot \xi) dt dx \\
&  \qquad =\int_{D}\int_0^1 \rho(t, x)( -H(x, \xi) + \frac{m}{\rho} \cdot \xi)  dtdx  \\
&    \qquad = \int_{D}\int_0^1 \rho(t, x)L(x, m/\rho)dtdx- \int_{D}\int_0^1 \rho(t, x)d_H(m/\rho, \xi)dtdx,  
\end{align*}
where 
$$
d_H(v, \xi)=L(x, v)+H(x, \xi)-v\cdot \xi. 
$$
By the Legendre transform we know that $d_H(v, \xi)\geq 0$, and $d_H(v, \xi)=0$ if any only if 
$$
\xi=\nabla_v L(x, v).
$$
On the other hand, the right hand side can indeed be achieved by taking 
$$
\xi=\nabla_v L(x, m/\rho), \quad \sigma=-H(x,\xi).
$$
This proves our claim (\ref{gl}).  This when combined with Step 1 gives 
$$
I_4= \inf_{\rho, m} J \geq C_L(\rho_0, \rho_1).
$$
\noindent{\bf  Step 3.} \\ 
Introducing a Lagrangian multiplier $\phi$ for the continuity equation $\partial_t \rho +\nabla_x \cdot m=0$,
and integrating by parts in $t$  and $x$, we obtain the Lagrangian
\begin{align*}
\mathcal L(\rho, m, \phi)
 & =J(\rho, m) +\int_0^1\int_D \phi (\partial_t \rho +\nabla_x \cdot m)dxdt \\ 
 & =  \int_{[0, 1]\times D } 
(L(m/\rho) -  m/\rho \cdot  \nabla_x \phi(t, x) -  \partial_t \phi(t,x)) \rho(t,x)dxdt \\
& \qquad  + \int_D (\phi(1, x)\rho_1(x)-\phi(0, x)\rho_0(x))dx.
\end{align*}
Note that 
\begin{align*} 
\inf_{\rho, m} J(\rho, m)  & = \inf_{\rho, m} \sup_\phi  \mathcal L(\rho, m, \phi) \\
&  \geq  \sup_\phi   \inf_{\rho, m}  \mathcal L(\rho, m, \phi) \\
&  = \sup_\phi \inf_\rho 
\left\{  \int_D (\phi(1, x)\rho_1(x)-\phi(0, x)\rho_0(x))dx  \right. \\
& \qquad \left.  - \int_{[0, 1]\times D } \rho ( \partial_t \phi + H(x, \nabla_x \phi) dx dt 
\right\} \\
& =\sup_\phi K(\phi),
\end{align*}

 where 
 $$
 K(\phi):=\left\{\int_{D } (\phi(1, x)\rho_1(x)-\phi(0, x)\rho_0(x))dx, \; \partial_t \phi + H(x, \nabla_x \phi) \leq 0
 \right\}.
$$
When sufficiently smooth primal and dual optimizers exist, the equality can also be verified directly.  
Let $(\rho, \phi)$ be a smooth solution to the following  
\begin{align*}
& \partial_t \rho +\nabla_x \cdot (\rho \nabla_p H(x, \nabla_x \phi))=0,  \\
&  \partial_t \phi(t, x) + H(x, \nabla_x \phi) = 0, \\
&  \rho(0,x)=\rho_0(x), \rho(1,x)=\rho_1(x),
\end{align*}
then equality holds in the Fenchel inequality.  The above formal optimality system should not be interpreted as an independent proof of strong duality; rigorous strong duality follows from the Fenchel–Rockafellar framework under the stated functional-analytic assumptions.

Combining this identity with the continuity equation gives
\begin{align*}
J(\rho, m) =& \int_0^1\int_{D } L(x, \frac{m}{\rho})\rho dx dt \\
= & \int_0^1\int_{D } \left( 
\nabla_x \phi \cdot \nabla_p H(x, \nabla_x \phi) -H(x, \nabla_x \phi(t, x))\right)\rho dx dt \\
=& \int_0^1\int_{D } \left( 
- \phi \nabla_x \cdot(\rho \nabla_p H(x, \nabla_x \phi))  + \rho \partial_t \phi(t, x)\right) dx dt\\
=& \int_0^1\int_{D }(\phi \partial_t \rho + \rho \partial_t \phi)dxdt  \\
=& \int_0^1\int_{D } \partial_t \left(\phi(t, x)\rho \right) dx dt \\
=& \int_{D } (\phi(1,x) \rho_1(x) -\rho(0,x)\rho_0(x))dx. 
\end{align*}
Thus, in the smooth setting, the primal and dual formulations attain the same value: 
\begin{equation}\label{JK}
\inf_{\rho, m} J=\sup_{\phi}K(\phi).
\end{equation}
Under suitable functional-analytic assumptions, strong duality follows from the Fenchel–Rockafellar duality theorem.

\noindent{\bf  Step 4.}  \\
Finally, we show that  
 $$ \inf_{\rho, m}  J = C_L(\rho_0, \rho_1).
 $$    
Let $X(t;  x, y), t\in [0, 1]$, denote a minimizing trajectory joining  $x$ to $y$ so that  
$$
c_L(x, y)=\int_0^1 L(X (t; x, y), \partial_t X (t; x, y))dt
$$
If $\pi \in \Pi(\rho_0, \rho_1)$  is an optimal transport plan, then    
$$
C_L(\rho_0, \rho_1)=\int_{D \times D} c(x, y)d\pi(x, y).
$$
The path family $X(t;x,y)$ induces the space-time density and momentum
\begin{align*}
& \rho(t, x)=\int_{D\times D} \delta(x-X(t;x, y))d\pi(x, y), \\
& m(t, x)=\int_{D\times D} \partial_t X (t; x, y) \delta(x-X(t;x, y))d\pi(x, y).
\end{align*}
One readily verify that $(\rho, m)$ satisfies the continuity equation and is therefore admissible for the functional $J$.   
Let $(\sigma, \xi)$ satisfy  
$$
 p: =\sigma+H(X, \xi) \leq 0.
 $$  
 Using the Legendre identity we have 
 $$
 L(X, v)=v\cdot \xi -H(X, \xi)+D_H(v, \xi),
 $$
 where 
 $$
 D_H(v, \xi) =  L(X, v) -v\cdot \xi +H(X, \xi) -v\cdot \xi \geq 0,
 $$ 
we obtain 
\begin{align*}
C_L(\rho_0, \rho_1) &=\int_0^1 \int_{D\times D} L(X, \partial_t X) d\pi dt\\
 & = \int_0^1\int_{D \times D} (\partial_t X \cdot \xi - H(X, \xi) + d_H(\partial_t X, \xi) d\pi dt \\
 & =  \int_0^1 \int_{ D \times D} \left( \sigma(X) + \partial_s X \cdot \xi(X) \right) d\pi ds \\
 & \quad +   \int_0^1\int_{ D \times D} \left(d_H(\partial_t X, \xi) - p \right)d\pi dt,
 \end{align*}
 where $p=\sigma+H(X, \xi)\leq 0$. Since $D_H \geq 0$ and $-p\geq 0$, it follows that  
$$
 C_L(\rho_0, \rho_1)  \geq \int_0^1\int_D \sigma \rho +\xi\cdot m dx ds. 
$$
Taking the supremum over all admissible pairs $(\sigma, \xi)$ yields, 
$$
C_L(\rho_0, \rho_1) \geq J(\rho, m) \geq  \inf_{\rho, m} J(\rho, m).
$$
Combined with the saddle-point identity established earlier, 
$$
C_L(\rho_0, \rho_1)=\sup_{\phi} K(\phi).
$$
To prove the reverse inequality, we invoke the Kantorovich duality formula
$$
C_L(\rho_0, \rho_1)= \sup_{\phi_0, \phi_1} \left\{ \int_{D } (\phi_1 d\rho_1- \phi_0 d\rho_0) \right\},
$$ 
where  the admissible Kantorovich potentials  $\phi_1$ and $\phi_0$ are continuous functions on $D $ satisfying 
$$
\phi_1(y)-\phi_0(x) \leq c_L(x, y), \; \forall x, y \in D .
$$
Suppose now that $\phi(s, x)$ is continuously differentiable and satisfies 
$$
\partial_t \phi +H(x, \nabla_x \phi)\leq 0.
$$
For any $C^1$ curve $\gamma(t): [0, 1] \to D $ connecting $x$ and $y$,  
\begin{align*}
 \int_0^1 L(\gamma, \dot \gamma)dt
 & \geq \int_0^1  \nabla_x\phi(t, \gamma(t))\cdot \dot \gamma(t) -H(\gamma(t), \nabla_x \phi(t, \gamma(t)) )dt\\
& \geq \int_0^1(\nabla_x\phi(t, \gamma(t))\cdot \dot \gamma(t) + \partial_t \phi(t, \gamma(t)))dt\\
& =\phi(1, \gamma(1)) -\phi(0, \gamma(0)).
\end{align*}
Taking the infimum over all connecting curves gives
$$
\phi(1, y)-\phi(0, x) \leq c_L(x, y),
$$
so that $(\phi(0, x), \phi(1, y))$ is an admissible Kantorovich pair. Consequently,  
$$
C_L(\rho_0, \rho_1) \leq \sup_\phi K(\phi).
$$
Combining the two inequalities yields 
$$
C_L(\rho_0, \rho_1) = \inf_{\rho, m} J = \sup_\phi K(\phi).
$$
\end{proof}

We emphasize that the argument in Step 3 is formal rather than a complete functional-analytic proof of strong duality. 
To avoid lengthen technical details, we state the following result under explicit assumptions.

\begin{prop}
Assume that $L(x,\cdot)$ is proper, lower semicontinuous, convex, and superlinear, with the perspective functional $(\rho, m) \to \rho L(x, m/\rho)$ lower semicontinuous on the space of finite measures. Assume further that the endpoint densities admit a finite-cost admissible interpolation satisfying the constraint qualification of the Fenchel–Rockafellar theorem. Then the dynamic Lagrangian transport problem admits the dual representation
$$
\inf_{\rho, m} J=\sup_{\phi}K(\phi).
$$
In particular, there is no duality gap.
\end{prop} 

\begin{rem}
(1) When sufficiently smooth optimizers exist, equality in the Fenchel inequality yields 
$$
m/\rho=\nabla_p  H(x,\nabla_x \phi), 
 $$
and the dual constraint is saturated, giving the associated Hamilton–Jacobi/continuity optimality system. The existence of such smooth optimizers is a separate regularity question and is not required for the duality result. \\
(2) The dynamic optimal transport  literature routinely formulates the dual in terms of Lipschitz/viscosity Hamilton–Jacobi subsolutions rather than assuming classical $C^1$ potentials \cite{IL24}. General convex-analytic treatments of nonlinear optimal transport also emphasize the measure-theoretic/functional-analytic formulation and optimality conditions \cite{PP19}. 
\end{rem}

\noindent{\bf Definition of $I_6$.} \\
 Under the assumption of $L(x, \cdot)$ is strongly convex, we may introduce the equivalent Hamiltonian formulation 
  \begin{align*}
 I_6 := \inf_{\rho, \phi}  \int_0^1\int_{D} L(x, \nabla_p H(x, \nabla_x \phi))\rho dx dt, 
  \end{align*}
  subject to the continuity equation 
 $$
 \partial_t \rho +\nabla_x (\rho \nabla_p H(x, \nabla_x \phi))=0, 
 $$
together with  the endpoint conditions 
 $$
  \rho(0, \cdot)=\rho_0, \; \rho(1, \cdot)=\rho_1.
 $$
 The following result shows that the optimal transport problem induced by a Lagrangian cost admits an equivalent Hamiltonian formulation in terms of the density $\rho$ and the 
 Hamilton–Jacobi potential $\phi$.  
 \begin{thm} Let $\rho_0, \rho_1 \in \mathcal P(D)$ be smooth,  strictly positive probability densities. Under the strong convexity of $L(x, \cdot)$, 
  \begin{align}\label{lf}
 C_L(\rho_0, \rho_1)= I_6.
  \end{align}
\end{thm}
\begin{proof}
Introduce the momentum variable 
$$
m=\rho v. 
$$  
Then the Eulerian formulation can be written as
$$
\inf_{\rho, m}J(\rho, m),
$$
where $\rho, m$ are smooth, $\rho>0$, and satisfy 
$$
\partial_t \rho +\nabla \cdot m=0, \quad \rho(0, \cdot)=\rho_0, \quad \rho(1, \cdot)=\rho_1. 
$$
It is immediate that this formulation is no larger than the Hamiltonian formulation, since every admissible pair $(\rho \phi)$ defines an admissible momentum through
$$
m=\rho \nabla_p H(x, \nabla_x \phi). 
$$
To prove the converse inequality, let $(\rho, m)$ be an admissible  pair  in the momentum formulation.  
Since $\rho$ is strictly positive, for each fixed time  $t\in [0, 1]$, we solve the elliptic problem 
  \begin{align}\label{lu1}
 \nabla_x\cdot  (m-\rho \nabla_pH(x, \nabla \phi))=0, \quad x\in  D , 
 \end{align}
 together with the periodic boundary condition $\partial D $. 
 
 The strict convexity of $H(x, \cdot)$, inherited from the convexity of the Lagrangian $L(x, \cdot)$ guarantees ellipticity of this problem. 
 Since $(\rho, m)$ depends smoothly on time, standard elliptic regularity implies that $\phi$ is also smooth in both space and time. 

Multiplying the elliptic equation by $\phi$ and integrating by parts yields
   \begin{align}\label{lu2}
 \int_D (m-\rho \nabla_pH(x, \nabla \phi))\cdot \nabla_x \phi dx=0. 
 \end{align}
 Next, applying the Fenchel–Young inequality,
$$
L(x, v)+H(x, p) \geq p\cdot v,
$$ 
with equality if and only if 
$$
v=\nabla_p H(x, p),
$$
gives 
\begin{align*} 
 \int_{D } L(x, m/\rho)\rho dx & \geq  \int_{D } (m \cdot \nabla_x \phi -  \rho H(x, \nabla \phi) dx \\
&   = \int_{D }  \rho (\nabla_p H(x, \nabla_x \phi) \cdot \nabla_x \phi -  H(x, \nabla_x \phi))dx,  
\end{align*}
where the second equality follows from the orthogonality identity above.

Applying the Legendre transform once again,
$$
\nabla_p H(x, p)\cdot p-H(x, p)=L(x, \nabla_p H(x, p)),
$$
we obtain 
$$
 \int_{D } L(x, m/\rho)\rho dx  \geq \int_D \rho L(x, \nabla_p H(x, \nabla_x \phi))dx. 
$$
Finally, integrating over time shows that every admissible pair $(\rho,m)$ 
produces an admissible pair $(\rho,\phi)$ with no larger action. 
Therefore the two variational formulations have the same infimum, establishing the desired equivalence.
\end{proof}

\section{Examples} 

In this section, we demonstrate that several  well-known example of cost functions.  \\ 

\noindent {\bf Example 1. (Quadratic cost)} \\
The standard quadratic cost 
$$
c(x, y)=|x-y|^2
$$
is recovered by choosing the Lagrangian $L(v)=|v|^2$. \\

\noindent {\bf Example 2 (Mechanical (kinetic +potential) Lagrangian)} \\
A classical example is the action functional from classical mechanics, 
$$
L(x, v)=\frac{1}{2}|v|^2 -V(x),
$$
where $V(x)$ is a potential energy, and the Hamiltonian is 
$$
H(x, p)=\frac{1}{2}|p|^2 +V(x). 
$$
This is perhaps the most classical example beyond the quadratic Wasserstein distance. \\ 

\noindent {\bf Example 3 (Convex translation-invariant cost)}\\
More generally, consider   
$$
c(x, y)=h(x-y),
$$ 
where $h: \mathbb{R}^d \to \mathbb{R}$ is a smooth,   convex function.  This corresponds to a Lagrangian of the form
$$
L(v)=h(v),
$$ 
whose associated Hamiltonian is given by the Legendre transform of $h$.  The minimizing trajectory between $x$ and $y$ is the straight-line path 
$$
\omega(t)=x+t(y-x),
$$ 
and therefore 
$$
c(x, y)=\inf_\omega \int_0^1 h(\dot \omega(t))dt=h(x-y). 
$$
This observation appears, for example, in \cite{Br03, BB07} 
and is also implicit in Villani's treatment \cite{Vi08}. \\

\noindent {\bf Example 4 (Anisotropic quadratic cost)} \\
Let $A(x)$ be a symmetric positive-definite diffusion tensor and consider the quadratic Lagrangian 
$$
L(x, v)=\frac{1}{2}v\cdot A^{-1}(x)v.
$$ 
The corresponding Hamiltonian is 
$$
 H(x, p)=\frac{1}{2} p\cdot A(x)p. 
 $$
 Because the Lagrangian is quadratic in the velocity, the associated action defines a Riemannian-type transport metric on the space of probability densities. This yields a weighted Wasserstein geometry that arises naturally in anisotropic Fokker--Planck equations and related gradient-flow formulations; see \cite{Li09, LZ19}.  \\
 
 %

\noindent {\bf Example 5 (Optimal control induced Lagrangian cost)} \\
Consider a controlled dynamical system   
$$
\dot z(t)=f(z(t), u(t)),
$$
with running cost $l(z, u)$.  Eliminating the control variable leads to the effective Lagrangian 
$$
L(x, v)=\inf_{u: f(z, u)=v} l(z, u),
$$
which represents the minimal cost required to generate velocity $v$ at position $x$. The associated transport cost is then given by
$$
c_L(x, y)=\inf_{z(0)=x, z(1)=y}  \int_0^1 L(z(t), \dot z(t))dt. 
$$
which is the value function of a deterministic optimal control problem and forms the basis of the Hamilton--Jacobi--Bellman theory 
\cite{BC97,Ev10}. When $L$ is a Tonelli Lagrangian, this action also induces an optimal transport cost \cite{BB07,Fa19}. \\

\noindent {\bf Example 6 (Continuous-depth neural networks and  optimal control).}
Continuous-depth neural networks, such as Neural ODEs,  can be viewed as controlled dynamical systems of the form 
\begin{align*}
    \dot z(t)
    =
    f(z(t),\theta(t)),\quad t\in [0, T],
\end{align*}
where $z(t)\in\mathbb{R}^d$ represents the hidden feature state and $\theta(t)$ denotes the time-dependent trainable network parameters. The training process can be viewed as an optimal control problem 
\begin{align*}
    \min_{\theta(\cdot)}
    \;
    \Phi(z(T))
    +
    \int_0^T
    R(\theta(t))\,dt,
\end{align*}
where $\Phi$ is the terminal loss and $R$ is a regularization functional penalizing the complexity or
energy of the parameters.

By eliminating the control variable leads to an effective Lagrangian describing the minimal cost required to generate a prescribed velocity $v$ at a given state $z$:
\begin{align*}
    L(z,v)
    =
    \inf_{\theta:\,f(z,\theta)=v}
    R(\theta).
\end{align*}
Consequently, the induced transition cost between two feature states $x$ and $y$
is given by the action minimization problem 
\begin{align*}
    c_L(x,y)
    =
    \inf_{\substack{z(0)=x\\ z(T)=y}}
    \int_0^T
    L(z(t),\dot z(t))\,dt.
\end{align*}
The terminal loss $\Phi(z(T))$  acts as an endpoint potential, leading to the optimization problem 
$$
 \inf_y \{c_L(x, y)+\Phi(y)\},
 $$ 
where $c(x,y)$ is the Lagrangian transport cost.  This naturally connects Neural ODEs, optimal control, and Lagrangian optimal transport \cite{EHL19, Ch18}. 

These examples illustrate the broad scope of the framework, covering important settings in optimal transport, convex analysis, differential geometry, classical mechanics, optimal control, and deep learning. They show that many seemingly different transport costs can be naturally interpreted within a unified Lagrangian formulation.
\section{Discussions} 

In this work, we studied optimal transport costs induced by a general Lagrangian cost from both the Lagrangian and Eulerian perspectives. We showed that the relationship between these viewpoints is substantially richer than previously recognized. In particular, we established several equivalent characterizations of the induced transport cost, including geometric, Lagrangian, Eulerian, convex optimization, Hamilton-Jacobi dual, and Wasserstein-Hamiltonian formulations, together with explicit correspondences among them  

These equivalent formulations provide a unified variational framework for optimal transport with general Lagrangian costs. They reveal the intrinsic connections among transport maps, minimizing trajectories, Hamiltonian dynamics, convex duality, and variational principles, thereby extending the classical Benamou--Brenier dynamic formulation beyond the quadratic-cost setting \cite{BB00, Vi03, Vi08}. 
Moreover, the dual formulation naturally yields a coupled Hamilton--Jacobi/continuity system, providing a PDE characterization of the optimal transport problem and exposing its underlying Hamiltonian structure on the Wasserstein space \cite{Ot01, AG08}.

This Hamiltonian viewpoint is closely related to several active research directions. Similar PDE systems arise in mean-field games 
\cite{LL07}, optimal control, and interacting particle systems, while Wasserstein--Hamiltonian flows provide a natural infinite-dimensional extension of classical Hamiltonian mechanics \cite{AGS05, CLZ20}.  More recently, related mean-field formulations have also appeared in the mathematical analysis of machine learning, including over-parameterized neural networks, mean-field optimization, and transformer dynamics \cite{MMN18, CB18, RV22, EHL19}.  In these settings, probability measures, rather than individual particles or parameters, become the fundamental dynamical variables.

Beyond its theoretical interest,  the equivalent variational formulations offer multiple perspectives for designing numerical algorithms, including particle methods, PDE-based discretizations, convex optimization approaches, and neural variational solvers.  We hope that these results will provide a useful perspective for further studies of optimal transport with general Lagrangian costs, as well as its connections to geometric mechanics, mean-field dynamics, scientific machine learning, and optimal control.

\bigskip
\subsection*{Acknowledgement}  The author thanks Robert Pego (Carnegie Mellon University) for carefully reading an earlier version and for his insightful comments, which greatly improved the manuscript.

\bibliographystyle{abbrv}

\begin{thebibliography}{10}

 \bibitem{AG08} 
 L. Ambrosio and  W. Gangbo. 
\newblock Hamiltonian {ODE}s in the Wasserstein Space of Probability Measures.
\newblock Communications on Pure and Applied Mathematics,  61(1), 18--53, 2008. 

\bibitem{AGS05} 
L. Ambrosio, N. Gigli, and G. Savare.
\newblock Gradient Flows in Metric Spaces and in the Space of
Probability Measures.
\newblock {\em Lectures in Mathematics}, Birkh\"{A}auser Verlag, Basel, 2005.

\bibitem{Ar17} 
Arjovsky, Martin and Chintala, Soumith and Bottou, L{\'e}on.
\newblock Wasserstein Generative Adversarial Networks.
 \newblock Proceedings of ICML, 214--223, 2017. 

\bibitem{Br91}
Y. Brenier.
\newblock  Polar factorization and monotone rearrangement of vector-valued functions.
\newblock {\em Comm. Pure Appl. Math.}, 44:375--417, 1991.

\bibitem{BB00}
J-D Benamou and Y. Brenier.
\newblock  A computational fluid mechanics solution to the Monge--Kantorovich mass transfer problem.
\newblock {\em Numer. Math.},  84: 375--393, 2000.


\bibitem{BB07}
Patrick Bernard and Boris Buffoni.
\newblock Optimal mass transportation and Mather theory.
\newblock {\em J.  European Math.}, Soc. 9: 85--121,  2007.

\bibitem{BC97} 
Bardi, M. and Capuzzo-Dolcetta, I.
\newblock Optimal Control and Viscosity Solutions of Hamilton–Jacobi–Bellman Equations.
\newblock Birkhäuser, 1997.


\bibitem{Br03}
Y. Brenier.
\newblock  Extended {M}onge-{K}antorovich theory.
\newblock {\em Optimal transportation and applications ({M}artina {F}ranca},
Lecture Notes in Math, Springer, Berlin,  Volume 1813: 91--121, 2003.

\bibitem{CB18}
Lénaïc Chizat and Francis Bach.
\newblock  On the Global Convergence of Gradient Descent for Over-parameterized Models Using Optimal Transport.
\newblock  Advances in Neural Information Processing Systems (NeurIPS 2018).



\bibitem{CFT17} 
Courty, Nicolas and Flamary, R{\'e}mi and Tuia, Devis.
\newblock  Optimal Transport for Domain Adaptation,
\newblock  IEEE Transactions on Pattern Analysis and Machine Intelligence, 39(9),  1853--1867, 2017. 


\bibitem{Cu13}
Marco Cuturi.
\newblock Sinkhorn Distances: Lightspeed Computation of Optimal Transport.
 \newblock Advances in Neural Information Processing Systems, 26, 2292--2300, 2013.
 
\bibitem{Ch18} 
Chen, Ricky T. Q. and Rubanova, Yulia and Bettencourt, Jesse and Duvenaud, David.
\newblock Neural Ordinary Differential Equations. 
\newblock NeurIPS, 2018.


\bibitem{CLZ20} 
S.-N. Chow, W. Li, and H. Zhou.
\newblock  Wasserstein Hamiltonian flows.
\newblock J. Di!erential Equations,  268 (2020), pp. 1205–1219. 

\bibitem{Da08} 
  Dacorogna, Bernard.
  \newblock  Direct Methods in the Calculus of Variations.
 \newblock  Springer, 2nd edition, 2008 


\bibitem{EHL19} 
E. Weinan, Jiequn Han, and Qianxiao Li.
\newblock  A Mean-Field Optimal Control Formulation of Deep Learning.
 \newblock  Research in Mathematical Sciences, 6(10),  2019.


\bibitem{Ev89}
Evans, L.C. 
\newblock Partial differential equations and Monge--Kantorovich mass transfer.
\newblock {\em Lecture Notes}, 1989.

\bibitem{Ev10} 
Evans, L. C.
\newblock Partial Differential Equations.
\newblock  2nd ed., American Mathematical Society, 2010.

\bibitem{Fl18} 
Flamary, R{\'e}mi and Courty, Nicolas. 
\newblock  POT: Python Optimal Transport.
\newblock  Journal of Machine Learning Research, 18(78), 1--8, 2017. 

\bibitem{FF10}
A. Fathi and A. Figalli.
\newblock Optimal Transportation on Non-compact Manifolds.
\newblock {\em Israel Journal of Mathematics},  175(1):1--59,  2010.

\bibitem{FGY11}
A. Figalli, W. Gangbo, and T. Yolcu.
\newblock  A variational method for a class of parabolic PDEs.
\newblock {\em Annali della Scuola Normale Superiore di Pisa-Classe di Scienze-Serie},  10(1): 2011, 207--. 

\bibitem{Fa19} 
Albert Fathi. 
\newblock Weak KAM Theorem in Lagrangian Dynamics.
\newblock Cambridge University Press, 2019. 

  

  
\bibitem{GM96}
 W. Gangbo and R. J. Mccann.  
\newblock  The geometry of optimal transportation.
\newblock {\em Acta Math.}, 177:113--161, 1996.

\bibitem{IL24} 
Ishida, S., Lavenant, H. 
\newblock Quantitative Convergence of a Discretization of Dynamic Optimal Transport Using the Dual Formulation. 
\newblock Found Comput Math 26:  349–384, 2026. 

\bibitem{JL20} 
Jacobs, Matt and L{\'e}ger, Flavien. 
\newblock  A fast approach to optimal transport: The back-and-forth method
\newblock  Numerische Mathematik, 146(3): 513--544, 2020.

\bibitem{JLL21} 
Jacobs, Matt and Lee, Wonjun and L{\'e}ger, Flavien
\newblock  The back-and-forth method for Wasserstein gradient flows. 
\newblock  ESAIM: Control, Optimisation and Calculus of Variations, 27: 28, 2021.


\bibitem{JKO98} 
R. Jordan, D. Kinderlehrer, and F. Otto.
\newblock The variational formulation of the Fokker--Planck
equation.
\newblock {\em SIAM J. Math. Anal.},  29(1), 1--17, 1998.



\bibitem{Ka48}
L.V. Kantorovich. 
\newblock On a problem of Monge. 
\newblock {\em Uspekhi Mat. Nauk.},  3: 225--226, 1948.


\bibitem{KS84}
M. Knott and C.S. Smith. 
\newblock On the optimal mapping of distributions. 
\newblock {\em J. Optim.Theory Appl.},  43(1): 39--49, 1984. 

\bibitem{Li09} 
S. Lisini. 
\newblock  Nonlinear diffusion equations with variable coefficients as gradient flows in Wasserstein spaces. 
\newblock {\em ESAIM Control Optim. Calc. Var.}, 15(3): 2009, 712--740.

\bibitem{LL07} 
Lasry, Jean-Michel; Lions, Pierre-Louis.  \newblock Mean field games. 
\newblock {\em Japanese Journal of Mathematics}, 2 (1): 229–260, 2007.  
 
 
\bibitem{LZ19} 
H.  Liu and A. Tzavaras.  
\newblock  Variational structure of Fokker-Planck equations with a variable matrix mobility. 
\newblock  arXiv:2505.10676, 2025. 

\bibitem{Mc97}
R.J. McCann. 
\newblock A convexity principle for interacting gases. 
\newblock {\em Adv. Math.},  128(1):153--179, 1997. 

\bibitem{MMN18} 
Song Mei, Andrea Montanari, and Phan-Minh Nguyen.
\newblock  A Mean Field View of the Landscape of Two-Layer Neural Networks. 
\newblock  Proceedings of the National Academy of Sciences (PNAS), 115(33): E7665–E7671, 2018.


\bibitem{Ot01} 
F. Otto.
\newblock The geometry of dissipative evolution equations: the porous medium equation.
\newblock {\em Comm. Partial Differential Equations},  26 (1-2), 101--174, 2001.

 \bibitem{PC19} 
Gabriel Peyr{\'e} and Marco Cuturi. 
  \newblock Computational Optimal Transport.
  \newblock Foundations and Trends in Machine Learning, 11(5-6), p355-607, 2019.

 \bibitem{PP19} 
 Pennanen, Teemu and Perkki{\"o}, Ari-Pekka. 
 \newblock Convex duality in nonlinear optimal transport. 
  \newblock ournal of Functional Analysis, 277(4): 1029--1060, 2019. 
  
  
 title   = {Convex duality in nonlinear optimal transport},
  author  = {Pennanen, Teemu and Perkki{\"o}, Ari-Pekka},
  journal = {Journal of Functional Analysis},
  volume  = {277},
  number  = {4},
  pages   = {1029--1060},
  year    = {2019},
  

\bibitem{RR98} 
S.T. Rachev and L. R\"{u}schendorf.  
\newblock Mass transportation problems. Vol. I and II. 
\newblock Probability and its Applications, Springer, New York, 1998.

\bibitem{Ru98} 
Rubner, Yossi and Tomasi, Carlo and Guibas, Leonidas J. 
\newblock The Earth Mover's Distance as a Metric for Image Retrieval.
\newblock International Journal of Computer Vision, 40, 99--121, 1998.
 
\bibitem{RV22} 
Grant M. Rotskoff and Eric Vanden-Eijnden,
\newblock  Neural Networks as Interacting Particle Systems: Asymptotic Convexity of the Loss Landscape and Universal Scaling of the Approximation Error.
\newblock  Communications on Pure and Applied Mathematics, 75(1): 188–248, 2022.



\bibitem{Vi03} 
C. Villani.
\newblock Topics in optimal transportation, Graduate Studies in Mathematics.
 \newblock {\em Vol. 58, American Mathematical Society}, Providence, R.I., 2003.


\bibitem{Vi08} 
C. Villani.
\newblock Optimal Transport: Old and New.
 \newblock {\em  Vol. 338, Springer Science \& Business Media}, 2008. 
 


\end{thebibliography}

\end{document}